\documentclass[12pt,reqno]{amsart}

\usepackage{amsmath}
\usepackage{amsfonts}
\usepackage{amssymb}
\usepackage{verbatim}
\usepackage[left=2.9cm,right=2.9cm,top=3cm,bottom=3cm]{geometry}

\usepackage[dvipsnames]{xcolor}
\colorlet{red}{RubineRed!70}

\usepackage{enumerate}

\usepackage{graphicx}

\DeclareMathOperator{\lcm}{lcm}

\newcommand{\N}{\mathbb N}

\newtheorem{theorem}{Theorem}

\newtheorem{rem}{Remark}

\newtheorem{prop}{Proposition}

\newtheorem{example}{Example}
\newtheorem{definition}{Definition}
\newtheorem{corollary}{Corollary}

\title[Concordant voting situations]{Existence and construction of voting situations concordant with ranking 
patterns}

\author{Emilio De Santis}
\address{University of Rome La Sapienza, Department of Mathematics
Piazzale Aldo Moro, 5, 00185, Rome, Italy}
\email{desantis@mat.uniroma1.it}

\author{Fabio Spizzichino}
\address{University of Rome La Sapienza
Piazzale Aldo Moro, 5, I-00185, Rome, Italy}
\email{fabio.spizzichino@fondazione.uniroma1.it}

\begin{document}
\begin{abstract}
 Referring to a standard context of voting theory, and to the classic notion of \textit{voting situation}, here we show that it is  possible to observe any arbitrary set of elections' outcomes, no matter how paradoxical it may appear. On this purpose we
 use results, presented in the recent paper \cite{DS22}, that hinge on the concept of ranking pattern concordant with a probability model for non-negative random variables and on a related role of special load-sharing models.
Our results here will be obtained by suitably extending those therein, and by converting them into the context of voting.


\end{abstract}

\maketitle
\section{Introduction}
Referring to the classic notion of \textit{voting situation} which arises
in voting theory, we show that it is possible to observe any
arbitrary family of voting outcomes, no matter how paradoxical it may appear.
On this purpose we use recent results, obtained in the paper \cite{DS22} and concerning with
a multivariate extension of \textit{stochastic precedence.} Such results will
be  extended and converted into the context of voting. Our conclusions are conceptually similar
to those coming from seminal results given by Saari (see e.g. \cite{Saari(dictionary)}, \cite{Saari(SIAM)}).
However our method is based on special multivariate probabilistic models of
survival analysis and emphasis is given here to issues related to constructive
aspects of the desired voting situations.

As a reference point, we consider the standard scenario of voting that will be
briefly surveyed as follows. For more complete descriptions and presentations,
Readers are addressed, e.g., to \cite{Nurmi}, \cite{GL2017},
\cite{BachmeierEtAl}, \cite{MontesEtAl20}, \cite{Stearns} and bibliography cited therein.

We think of a system of possible elections, based on a fixed set of potential
\textit{candidates,} labeled as $1,2,...,m$, and on a fixed set of
\textit{voters, }$v_{1},v_{2},...,v_{n}$. Denote by $\left[  m\right]
\equiv\{1,2,...,m\}$ the set of candidates and by $V$ $\equiv\{v_{1}%
,v_{2},...,v_{n}\}$ the set of voters.

Each voter is supposed to cast one and only one vote, at
any election. On the other hand different elections can be considered, respectively characterized by the
different subsets $A\subseteq\left[  m\right]  $ of participating candidates.
The differences among elections' outcomes, encountered at varying
the subsets of candidates, are just at the core of the questions considered here.

Concerning the mechanism determining the vote of the single voter $v_{l}$
($v_{l}\in V,l=1,...,n$) in any of such elections, it is assumed that the
individual preferences, among candidates, of $v_{l}$ are \textit{complete},
\textit{transitive,} and indifference is not allowed between any two
candidates. Thus, individual preferences give rise to a \textit{linear
preference ranking}. In other words the preference ranking of $v_{l}$ is
described by a permutation \emph{$r_{l}$ }over the set $[m]$. Such a
permutation is to be viewed as a piece of a-priori information, established
independently of the special subsets of candidates which will be encountered
in the elections. Thus no decision strategy is needed and, in an election
characterized by a subset $A\subseteq\left[  m\right]  $ of candidates, the
vote expressed by $v_{l}$ is just addressed to the one who is the preferred
candidate among the elements of $A$, according to the permutation $r_{l}$, for
$l=1,...,n$.

Looking at the collective of all the $n$ voters, denote by $N^{(m)}%
(j_{1},\ldots,j_{m})$ the number of those $v_{l}$'s who share the same linear
preference ranking $\left(  j_{1},\ldots,j_{m}\right)  $, i.e. such that
$r_{l}=(j_{1},\ldots,j_{m})$. We focus attention on the family of numbers
$\mathcal{N}^{(m)}\equiv\{N^{(m)}(j_{1},\ldots,j_{m})\}$, which is typically
referred to as the \textit{voting situation}. For $A\subseteq\left[  m\right]
$ and $j\in A$, we denote by $n_{j}(A)$ the number of votes obtained by candidate
$j$ in an election where $A$ is the set of competing candidates. The numbers
$n_{j}(A)$, for $A\subseteq\left[  m\right]  $ and $j\in A$, are clearly
determined by the voting situation $\mathcal{N}^{(m)}$.  In the analysis of the family of numbers $n_{j}(A)$, for $A\subseteq\left[
m\right]  $, $j\in A$, several types of paradoxes can be encountered.

It is well-known that, as a starting point of voting theory, some voting
situations give rise to \textit{Condorcet paradoxes, }namely to
non-transitivity of collective voting outcomes.

The literature  related  to voting paradoxes has a long history (see e.g. \cite{AlonPara}, \cite{Fishburn(1981)}, \cite{Mala99}, \cite{Saari(dictionary)}, \cite{Saari(SIAM)}, \cite{Saari(handbook)}). Along such
literature a very wide catalogue of, more or less complex, special cases has
been thoroughly analyzed to show that the possibility of some types of
paradoxes cannot be excluded. Other parts of the literature have been devoted
to analyzing the probability that paradoxes can manifest, in the case of random voting situations and under different
stochastic models (see e.g. \cite{HMRZ20}).

Essentially, we are not going to deal with random voting situations and with the probability of observing
paradoxical outcomes. Rather we hinge on the probabilistic method developed in
\cite{DS22} to show that, for any paradox emerging in  the analysis of the numbers $n_j(A)$, it is possible to construct voting situations which give rise to it.

A clearer idea about our purposes can be summarized as follows.

In order to describe formally outcomes of the elections for different subsets
of candidates, to any voting situation we associate a \textit{q-concordant}
\textit{ranking pattern,} according to the definitions that have been
introduced in \cite{DS22} and that will be recalled below (Definitions \ref{RankFunctns}, \ref{RankPattExdef1},
and \ref{def2}).

The notion of ranking pattern extends the one of \textit{majority graph}  and can be seen as an ordinal variant of a
choice function. Within the family of all the ranking pattern we distinguish the subclass of \textit{strict} ranking patterns.

Several different voting situations can be q-concordant with a same (strict or not)
ranking pattern.
In \cite{DS22}, the definition was given of ranking pattern p-concordant with a multivariate survival model and it has been proven that, for any strict ranking pattern, there exist p-concordant survival models (see 
Theorem 2 in \cite{DS22}, recalled here as Theorem~\ref{EpsilonEspliciti} in Appendix).
Starting from this result, here we prove the existence of q-concordant voting situations for any (strict or weak) ranking pattern (Theorem \ref{forti} and Theorem \ref{forti-deboli}).
We furthermore analyze different aspects concerning with related constructions and with possible numbers of voters.
The method of proof for Theorem 2 in \cite{DS22} has been based on a special class of load-sharing models and it is completely constructive.
The latter circumstance will in particular allow us to detail other aspects about the difference between strict and weak
ranking patterns, for what concerns possible numbers of voters associated to q-concordant voting situations.

More precisely, the plan of the paper is as follows.

In Section \ref{sec2}, we recall notation, definitions, and some properties of
quantities related with the multivariate extension of stochastic precedence, as
studied in \cite{DS22}. In particular, we recall the definition of ranking pattern p-concordant with a probability model for a $m$-tuple of non-negative random variables. Based on this material we will in particular
detail the identity between the voting-theory scenario and multivariate
stochastic precedence.

Section  \ref{section 3} is devoted to presenting results concerning the construction of
voting situations. By exploiting and extending results presented in
\cite{DS22}, we will in particular single out a method to construct
voting situations concordant with arbitrary  ranking patterns, by starting from
voting situations concordant with strict ranking patterns (see Theorem \ref{forti-deboli} and the related proof). 

 Distinguishing between the two cases of strict and weak ranking patterns, the issue of possible values of the
number of voters required for obtaining  q-concordant voting situations is studied in Section  \ref{sec4}.

In Section  \ref{sec5} we present a discussion based on some concluding remarks and some examples. In particular we point out how Theorem \ref{forti-deboli}  leads us  to a general existence result also concerning with concordant survival models. Such a generalization, that fills a gap present in \cite{DS22},  would not be smooth without the passage through the language of voting theory.

For the sake of self-consistency, in the Appendix we review definitions and main properties
of load-sharing models, which are used in the construction of survival
models p-concordant with strict ranking patterns.
We finally present a proof for a related existence result (Theorem \ref{th1semplificato}). Such a proof is independent from the one of the
corresponding constructive result presented in \cite{DS22} and it rather aims to show the
 reasons why solutions can be found within the family of
load-sharing probabilistic models.

\section{Notation, Definitions, and Preliminary Results}
\label{sec2}
 Let $[m]=\{1,...,m\}$ and denote by $\Pi_{m}$ the set of
permutations of the elements of $[m]$.

\noindent Let $X_{1},\dots,X_{m}$ be non-negative random variables satisfying
the no-tie condition
\begin{equation}
\mathbb{P}(X_{i}\neq X_{j})=1, \label{no-tie}%
\end{equation}
for $i,j\in[ m]$ with $i\neq j$. Let the symbols $X_{1:m},\ldots,X_{m:m}$
denote the order statistics of $X_{1},\dots,X_{m}$ and let $\mathbf{J}%
\equiv\left(  J_{1},\ldots,J_{m}\right)  $ denote the random vector defined by
the position
\begin{equation}
J_{r}=i\Leftrightarrow X_{i}=X_{r:m} \label{random}%
\end{equation}
for any $i,r\in[ m]$. \ For $(j_{1}, \ldots, j_{m} ) \in\Pi_{m}$, we set
\begin{equation}
p_{m}(j_{1},\ldots,j_{m})=\mathbb{P}(J_{1}=j_{1},\,J_{2}=j_{2},\,\ldots
,\,J_{m}=j_{m}), \quad\forall(j_{1}, \ldots, j_{m} ) \in\Pi_{m}
\label{pTotali}%
\end{equation}
and by $\mathbf{P}_{\mathbf{J}}$ we briefly denote the collection of all the
$m!$ probabilities in \eqref{pTotali}. Clearly, $\mathbf{P}_{\mathbf{J}}$ is
determined by the joint probability distribution of $X_{1},\dots,X_{m}$ \ and,
in its turn, it determines a probability distribution over the finite set
$\Pi_{m }$. Consider also the family
\[
\mathcal{A}_{( m )} :=   \{\alpha_{j}(A);A\subseteq\lbrack m],j\in A\},
\]
where $\alpha_{j}(A)$ denotes the probability defined by
\begin{equation}
\alpha_{j}(A):=\mathbb{P}(X_{j}=\min_{i\in A}X_{i}). \label{alphasNv}%
\end{equation}
In some contexts it is useful to look at $\alpha_{j}(A)$ as a \textit{winning
probability}.

\begin{rem}
\label{RemADetermined}The family $\mathcal{A}_{(m)}$ is determined by
$\mathbf{P}_{\mathbf{J}}$. In this respect a detailed formula is shown in
Proposition 1 of \cite{DS22}.
\end{rem}

\bigskip

We are generally interested in comparisons of the type
\begin{equation}
\alpha_{i}(A)\geq\alpha_{j}(A)\text{, for }A\subseteq\left[  m\right]  ,i,j\in
A. \label{MainComparisonNv}%
\end{equation}

When $A\equiv\{i,j\}$, the inequality appearing in
(\ref{MainComparisonNv}) is just equivalent to \textit{stochastic precedence}
of $X_{i}$ with respect to $X_{j}$, a notion of interest in different fields
of applied probability. When the cardinality $|A|$ of $A$ is greater than two,
we can look at the comparison in (\ref{MainComparisonNv}) as a form of
multivariate stochastic precedence, see also \cite{DMS20}.


In order to describe preference rankings among the
elements of the subsets $A\subseteq\left[  m\right] $, we introduce ranking functions $\sigma(A,\cdot):A\rightarrow
\{1,2,\ldots,|A|\}$.  For $i,j\in A$, it
will be $\sigma(A,i)<\sigma(A,j)$ if $i$ \textit{precedes }$j$\textit{ in
}\emph{$A$} whereas $\sigma(A,i)=\sigma(A,j)$ when they are\textit{ equivalent
in}\emph{\textit{ }$A$}.
More precisely $\sigma(A,i)$ is equal to $1$ plus the
number of elements belonging to $A$ and preceding $i$. Consider for example
the case $m=10$, $A=\{3,4,6,8,9\}$, where $6$ is the favorite element,
followed by $4$ and $8$ with equal merit, then followed by $3$ and finally
with $9$ as the less favorite element in $A$. Then
\[
\sigma(A,6)=1; \quad \sigma(A,4)=\sigma(A,8)=2;
\]
\[
\sigma(A,3)=4; \quad \sigma(A,9)=5.
\]
Formally, we  give the following
\begin{definition}
\label{RankFunctns}
For $A\subseteq\left[  m \right]  $, a mapping $\sigma(A,\cdot):A\rightarrow
\{1,2,\ldots,|A|\}$ is a \textit{ranking function} if it satisfy the following condition
\begin{equation}\label{constrain}
\sigma(A,i) =1+ \sum_{j \in A} \mathbf{1}_{ [ \sigma(A,i) -1] }(\sigma (A, j))
\end{equation}
for each $i \in A$.
\end{definition}
When some equivalence holds between two elements of $A$, we say that
$\sigma(A,\cdot)$ is a \textit{weak ranking function}.
When, on the contrary, the values $\sigma(A,j)$,
$j\in A$, are all different then  $\sigma(A,\cdot):A\rightarrow
\{1,2,\ldots,|A|\}$ is a bijective function, namely $\sigma(A,\cdot)$
describes a permutation of the elements of $A$. This case will be designated by the term  \textit{strict ranking function}. 

\begin{definition}
\label{RankPattExdef1}  For $m\geq2$, a \emph{ranking
pattern}\textit{ }over $[m]$ is a family of ranking functions
$\boldsymbol{\sigma}\equiv\{\sigma(A,\cdot):A\subseteq\left[  m\right] , \, |A|\geq 2 \}$. A
ranking pattern is said to be \emph{strict} when it does not contain any weak
ranking function. The symbols $\Sigma^{(m)}$ and $\hat{\Sigma}^{(m)}%
\subset\Sigma^{(m)}$ denote the collections of all the ranking patterns over
$[m]$ and of all the strict ranking patterns, respectively.
\end{definition}

In a natural way, one associates to the family $\mathcal{A}_{\left(  m\right)
}$ a ranking pattern $\boldsymbol{\sigma}$ defined as follows.

\begin{definition}
\label{def2} The ranking pattern $\boldsymbol{\sigma}\equiv\{\sigma
(A,\cdot): A\subseteq\left[  m\right] , \, |A|\geq 2 \}$ and the\emph{ $m$-}tuple\emph{
$(X_{1},\dots,X_{m})$} are $p$\emph{-concordant} whenever, for any
$A\subseteq\left[  m\right]  $ and $i,j\in A$ with $i\neq j$
\begin{equation}
\sigma(A,i)<\sigma(A,j)\Leftrightarrow\alpha_{i}(A)>\alpha_{j}(A),
\label{glialfa}%
\end{equation}%
\begin{equation}
\sigma(A,i)=\sigma(A,j)\Leftrightarrow\alpha_{i}(A)=\alpha_{j}(A).
\label{glialfauguali}%
\end{equation}

\end{definition}

Notice that the quantities $\sigma(A,i)$ are natural numbers belonging to
$[|A|]$, whereas the quantities $\alpha_{j}(A)$ are real numbers belonging to
$[0,1]$. Obviously for any $m$-tuple $(X_{1},\ldots,X_{m})$ there exists one
and only one concordant ranking pattern.

We also point out that,   for convenience, the concept of ranking function has been defined here

\bigskip

Let us now come back to the voting scenario described in the Introduction and
fix a voting situation $\mathcal{N}^{(m)}$. For $i\in\lbrack m]$ and
$\mathbf{j}\equiv(j_{1},\ldots,j_{m})\in\Pi_{m}$, \ the symbol $\phi
_{(j_{1},\ldots,j_{m})}(i)$ will be used to denote the position of $i$ within the
vector $(j_{1},\ldots,j_{m})$. Thus $\phi_{(j_{1},\ldots,j_{m})}(\cdot)$ is a
bijection on $[m]$. It is easy to check that, for any $A\subseteq\lbrack m]$
with $|A|\geq2$, the number $n_{i}(A)$ of votes obtained by $i$, in an
election where $A$ is the set of participating candidates and under
$\mathcal{N}^{(m)}$, can be expressed as follows
\begin{equation}
n_{i}(A)=\sum_{(j_{1},\ldots,j_{m})\in\Pi_{m}}N^{(m)}(j_{1},\ldots,j_{m}%
)\prod_{j\in A:j\neq i}\mathbf{1}_{\{\phi_{(j_{1},\ldots,j_{m})}%
(i)<\phi_{(j_{1},\ldots,j_{m})}(j)\}}.\label{voti1}%
\end{equation}
In fact all the voters who share the same linear preference ranking
$\mathbf{j}\equiv(j_{1},\ldots,j_{m})$ will cast their individual vote in favor of
$i$ if $i\in A$ is the favorite candidate within the set $A$, i.e. whenever
$\phi_{(j_{1},\ldots,j_{m})}(i)<\phi_{(j_{1},\ldots,j_{m})}(j)$, for any $j\in
A,$ with $j\neq i$.

We focus attention on the family
\begin{equation}
\mathbf{N}:=\{n_{j}(A);A\subseteq\left[  m\right]  ,j\in A\},\label{FamilyN}%
\end{equation}
associated to $\mathcal{N}^{(m)}$.  We also look at the total number of
voters:
\begin{equation}
n =n (  \mathcal{N}^{(m)} )  :=\sum_{(j_{1},\ldots,j_{m})\in\Pi_{m}%
}N^{(m)}(j_{1},\ldots,j_{m})\label{TotalNumberofVoters}
\end{equation}

Furthermore we point out that also the family $\mathbf{N}$ gives rise to a ranking pattern. In analogy with the above Definition \ref{def2},
it is in fact natural to introduce the following

\begin{definition}
\label{AnalogDefinition} A ranking pattern $\boldsymbol{\sigma}\equiv
\{\sigma(A,\cdot);A\subseteq\left[  m\right] , |A|\geq 2 \}$ and a voting situation
$\mathcal{N}^{(m)}$are $q$\emph{-concordant} \ whenever, for any
$A\subseteq\left[  m\right]  $ and $i,j\in A$ with $i\neq j$
\begin{equation}
\sigma(A,i)<\sigma(A,j)\Leftrightarrow n_{i}(A)>n_{j}(A), \label{glienne}%
\end{equation}%
\begin{equation}
\sigma(A,i)=\sigma(A,j)\Leftrightarrow n_{i}(A)=n_{j}(A). \label{glienneuguali}%
\end{equation}

\end{definition}

Keeping in mind that any voting situation $\mathcal{N}^{(m)}$ triggers a
corresponding family $\mathbf{N}$ and in order to detail the relation tying
$\mathbf{N}$ to $\mathcal{N}^{\left(  m\right) }$, we now single out a
discrete probability model, strictly related to $\mathcal{N}^{(m)}$ and
defined as follows. Consider the finite probability space $(\Pi_{m}%
,\mathbb{P}^{\mathcal{N}})$ with
\[
\mathbb{P}^{\mathcal{N}}(j_{1},\ldots,j_{m})=\frac{N^{(m)}(j_{1}%
,\ldots,j_{m})}{n},
\]
for any $(j_{1},\ldots,j_{m})\in\Pi_{m}$ and $n$ given in \eqref{TotalNumberofVoters}. On the space $(\Pi_{m}%
,\mathbb{P}^{\mathcal{N}})$ we define the no-tie, discrete, random variables
\begin{equation}
\widehat{X}_{i}(j_{1},\ldots,j_{m})=\phi_{(j_{1},\ldots,j_{m})}%
(i)\label{r.v.discreta}%
\end{equation}
for $i\in\lbrack m]$. Under such a position, the related discrete random
variables $(\widehat{J}_{1},\ldots,\widehat{J}_{m})$ are such that, for any
$(j_{1},\ldots,j_{m})\in\Pi_{m}$,
\[
p_{m}^{\mathcal{N}}(j_{1},\ldots,j_{m})=\mathbb{P}^{\mathcal{N}}%
((\widehat{J}_{1},\dots,\widehat{J}_{m})=(j_{1},\ldots,j_{m}))=
\]%
\begin{equation}
=\mathbb{P}^{\mathcal{N}}(j_{1},\ldots,j_{m})=\frac{N^{(m)}(j_{1}
,\ldots,j_{m})}{n}.\label{q=p}
\end{equation}
For what concerns the family $\widehat{\mathcal{A}}_{(m)}$, associated with
the random variables defined in \eqref{r.v.discreta}, one has
\[
\widehat{\alpha}_{i}(A)=\mathbb{P}^{\mathcal{N}}(\widehat{X}_{i}=\min_{j\in
A}\widehat{X}_{j})=
\]%
\begin{equation}
\sum_{(j_{1},\ldots,j_{m})\in\Pi_{m}}\frac{N^{(m)}(j_{1},\ldots,j_{m})}%
{n}\prod_{j\in A:j\neq i}\mathbf{1}_{\{\phi_{(j_{1},\ldots,j_{m})}%
(i)<\phi_{(j_{1},\ldots,j_{m})}(j)\}}=\frac{n_{i}(A)}{n}.\label{alfa1}%
\end{equation}


\medskip
The above identity allows us to compare the two settings of voting theory and
survival models, respectively concerning a voting situation $\mathcal{N}%
^{(m)}$ with $m$ candidates and the $m$ discrete random variables
$\widehat{X}_{1},...,\widehat{X}_{m}$. We can thus summarize as follows

\begin{prop} \label{BasicEq} Let $\mathcal{N}^{(m)}$ be a voting
situation and $\widehat{X}_{1},...,\widehat{X}_{m}$ be the random variables
defined in \eqref{r.v.discreta}. Then
\begin{equation}
\widehat{\alpha}_{i}(A)=\frac{n_{i}(A)}{n}.\label{fratelli}%
\end{equation}
for $A\subseteq\left[  m\right]  $ and $i \in A$.
\end{prop}

\bigskip

In view of the above Remark \ref{RemADetermined}, we notice that two models
sharing the same $\mathbf{P_{J}}$ also give rise to the same $\mathcal{A}%
_{(m)}$. By combining Definitions \ref{def2} and \ref{alphasNv} with
Proposition \ref{BasicEq}, one obtains an interesting conclusion concerning a
voting situation $\mathcal{N}^{(m)}$ and any random $m$-tuple $\left(
X_{1},\dots,X_{m}\right)  $ satisfying the relation
\begin{equation}
p_{m}(j_{1},\ldots,j_{m})=\frac{N^{(m)}(j_{1},\ldots
,j_{m})}{n},\quad    \forall (j_{1}
,\ldots,j_{m})\in\Pi_{m} \label{q=pBis}%
\end{equation}
More precisely, we can state the following result

\begin{prop}\label{EquivConcorda} Let $\mathcal{N}^{(m)}$ be a voting
situation and let $X_{1},\dots,X_{m}$ be no-tie, non-negative, random
variables such that the condition (\ref{q=pBis}) holds. Then the following two claims are equivalent:

a) the $m$-tuple $(X_{1},\dots,X_{m})$ is $p$-concordant with the ranking
pattern $\boldsymbol{\sigma}$;

b) the voting situation $\mathcal{N}^{(m)}$ is $q$-concordant with the ranking
pattern $\boldsymbol{\sigma}$.
\end{prop}
\bigskip

Even though an obvious remark, it is important to note that the condition
(\ref{q=pBis}) requires that all the probabilities $p_{m}(j_{1},\ldots,j_{m})$
are rational numbers, for $(j_{1},\ldots,j_{m})\in\Pi_{m}$.

\section{Existence of $q$-concordant voting situations}
\label{section 3}
It has been proven in \cite{DS22} that for any strict ranking pattern 
$\boldsymbol{\sigma}$ $\in\hat{\Sigma}^{(m)}$ there exists a probability model
 $X_{1},...,X_{m}$ which is $p$-concordant with $\boldsymbol{\sigma}$.
More precisely such a result, presented  therein as Theorem 2, shows the
existence of a $p$-concordant model within a special subclass of load-sharing
models ($LS(\boldsymbol{\varepsilon},\boldsymbol{\sigma
})$ models). The related statement is given in a quantitative form and the
proof therein is constructive. In this way, it provides us with an appropriate
basis for the arguments to be developed in the next section. Details will be
given in the Appendix, where its statement will be recalled (Theorem \ref{EpsilonEspliciti}) after briefly
reviewing necessary definitions and properties of load-sharing models.

For the purposes of the present section it is sufficient to hinge on the
existence result that will be also  stated and proven in the Appendix as
Theorem \ref{th1semplificato}. Its proof is independent from the one of Theorem 2  in \cite{DS22}.  Even though non-constructive, such a proof usefully
points out the  reasons why solutions can be found within the family
of probabilistic models of type $LS(\boldsymbol{\varepsilon},\boldsymbol{\sigma
})$ load-sharing. The vector $\boldsymbol{\varepsilon}$ can be chosen in such
a way that the quantities $p_{m}(j_{1},\ldots,j_{m})$ are rational  so that the
above Proposition \ref{EquivConcorda} can be applied.

Thus, while Theorem 2  in \cite{DS22} is a constructive result, we are also interested in Theorem \ref{th1semplificato} below which points out the
possibility of a larger choice in determining convenient constants
$\varepsilon(2),...,\varepsilon(m)$ giving rise to $p$-concordant load-sharing
models. Such a possibility can turn out to be useful in the search of
$q$-concordant voting situations with reduced numbers of voters.

First of all, by combining Theorem \ref{th1semplificato} and Proposition \ref{EquivConcorda},
it is possible to conclude that there exist voting situations
$q$-concordant with a given ranking pattern $\boldsymbol{\sigma}$ $\in\hat{\Sigma}^{(m)}$. In fact we state

\begin{theorem}
\label{forti}  For any $m \in\N$ and for any $\boldsymbol{\sigma}\in
\hat{\Sigma}^{(m)}$, there exists a voting situation $\mathcal{N}^{(m)}$ which
is $q$-concordant with $\boldsymbol{\sigma}$.
\end{theorem}

\begin{proof}
 By Theorem \ref{th1semplificato} and Remark \ref{razionali} one can select a
$L(\boldsymbol{\varepsilon},\boldsymbol{\sigma})$ model that is $p$-concordant
with $\boldsymbol{\sigma}$ and such that all the elements of the set
\[
\mathbf{P_{J}}=\{p_{m}(j_{1},\ldots,j_{m}):(j_{1},\ldots,j_{m})\in\Pi_{m}\}
\]
are rational. Let us now consider the numbers
\begin{equation}
N^{(m)}(j_{1},\ldots,j_{m})=np_{m}(j_{1},\ldots,j_{m}),\quad\forall
(j_{1},\ldots,j_{m})\in\Pi_{m},\label{Nintero}%
\end{equation}
where $n\in\N$ is a suitable constant such that all the $N^{(m)}(j_{1}%
,\ldots,j_{m})$'s are integers. The voting situation $\mathcal{N}^{(m)}$
defined by
\[
\mathcal{N}^{(m)}=\{N^{(m)}(j_{1},\ldots,j_{m}):(j_{1},\ldots,j_{m})\in\Pi
_{m}\}.
\]
is $q$-concordant with $\boldsymbol{\sigma}$ by Proposition
\ref{EquivConcorda}.
\end{proof}

Our next result will show how the existence of $q$-concordant voting
situations for any ranking pattern $\boldsymbol{\sigma}$ $\in\Sigma^{(m)}$ can
be obtained by taking as an assumption the existence of $q$-concordant voting
situations for the only strict ranking pattern in $\hat{\Sigma}^{(m)}$. As a consequence we can obtain a direct generalization of Theorem
\ref{forti}.

On this purpose, it is useful to introduce on the space of voting situations
the following three operations and to prove the ensuing related result.

\begin{itemize}
\item[1.] (Multiplication) For any $\ell\in\mathbb{N}$
\[
\ell\times\mathcal{N}^{(m)} = \{ \ell N^{(m)} (j_{1}, \ldots, j_{m}) : (j_{1},
\ldots, j_{m}) \in\Pi_{m}\}.
\]

\item[2.] (Addition) For any $\ell\in\mathbb{N}$
\[
\ell\text{\textcircled{+}}\mathcal{N}^{(m)} = \{ \ell+N^{(m)} (j_{1}, \ldots,
j_{m}) : (j_{1}, \ldots, j_{m}) \in\Pi_{m}\}.
\]

\item[3.] (Internal addition) For two voting situations $\mathcal{N}^{(m)} $ and
$\mathcal{N^{\prime}}^{(m)} $
\[
\mathcal{N}^{(m)} +\mathcal{N^{\prime}}^{(m)} =\{ N^{(m)}(j_{1}, \ldots, j_{m}
) + {N^{\prime}}^{(m)}(j_{1}, \ldots, j_{m} ) :(j_{1}, \ldots, j_{m} ) \in
\Pi_{m} \}.
\]

\end{itemize}

\begin{prop}
\label{trasforma} For any  $\ell \in \mathbb{N} $ the voting situations
$\ell \times \mathcal{N}^{(m)} $, $\ell$\emph{\text{\textcircled{+}}}$ \mathcal{N}^{(m)} $ and $ \mathcal{N}^{(m)} $,
are $q$-concordant with the same ranking pattern
$\boldsymbol{\sigma} \in \Sigma^{(m)}$.
\end{prop}
\begin{proof}
First we notice that, when $A$ is the set of candidates, the voting situation $\ell\times\mathcal{N}^{(m)} $
assigns $\ell \cdot n_{i} (A)$ votes to candidate $i$. Thus
 the  proportion  of votes assigned to each candidate remains the same as in the original voting
situation $\mathcal{N}^{(m)} $.  Hence $\ell\times\mathcal{N}^{(m)} $ and $\mathcal{N}^{(m)} $
are q-concordant with the same ranking pattern $\boldsymbol{\sigma}$.

We now prove that $\mathcal{N}^{(m)} $ and $\ell$\textcircled{+}$\mathcal{N}%
^{(m)} $ are q-concordant with the same $\boldsymbol{\sigma} \in\Sigma^{(m)}$.
Let us consider $A \subseteq[m] $ and two distinct elements $a,b \in A $. For
any given $(j_{1}, \ldots, j_{m} )\in\Pi_{m} $ let us consider the position of
$a$ and $b$, i.e. $\phi_{(j_{1}, \ldots, j_{m} )}(a) $ and $\phi_{(j_{1},
\ldots, j_{m} )}(b) $. We construct a new vector $(i_{1}, \ldots, i_{m})\in
\Pi_{m}$ by just interchanging the position of $a$ and $b$ and leaving unchanged the
positions of all the other components of the vector. By \eqref{voti1}, one has
that $N^{(m)} (j_{1}, \ldots, j_{m} )$ gives a positive contribution to $n_{a}
(A)$ if and only if $N^{(m)} (i_{1}, \ldots, i_{m} )$ gives a positive
contribution to $n_{b} (A)$. Therefore the cardinality of the set of vectors
that give a positive contribution to $n_{a} (A)$ coincides with the
cardinality of vectors that give a positive contribution to $n_{b} (A)$. This claim
concludes the proof because, under the voting situation $\ell$
\textcircled{+}$\mathcal{N}^{(m)} $, the numbers $n_{a} (A)$ and $n_{b} (A)$
are incremented by the same amount, i.e. $\ell$ multiplied by the cardinality
of the set
\[
\{ (j_{1}, \ldots, j_{m} )\in\Pi_{m} : \phi_{(j_{1}, \ldots, j_{m} )}(a)
\leq\phi_{(j_{1}, \ldots, j_{m} )}(i), \text{ for all } i \in A \}.
\]
\end{proof}
\bigskip

It is also useful to state the following result, whose proof is obvious.

\begin{prop} \label{prop+}
Let  $\mathcal{N}^{(m)}$ and $\mathcal{N'}^{(m)} $ be two voting situations. Then, for any $A\subseteq [m] $ and $i \in A$, the number of votes for $i$, under the
voting situation $\mathcal{N}^{(m)} +\mathcal{N'}^{(m)}$ and  when  $A$ is the set of  candidates, is equal to  $n_i (A)+n'_i (A)$.
\end{prop}

Now we proceed by proving that the claim in Theorem
\ref{forti} can be extended to the set $\Sigma^{(m)}$ of all possible ranking patterns.

\bigskip
On this purpose, we associate to a ranking pattern
$\boldsymbol{\sigma}\in\Sigma^{(m)}$ the index $\mathfrak{I}%
(\boldsymbol{\sigma})$ to be defined as follows.

For $\boldsymbol{\sigma}\in\Sigma^{(m)}$, $A\subseteq\lbrack m]$ with
$|A|\geq2$ and $\ell\in\lbrack|A|]$, consider the subset of $A$ defined by
\[
S(\boldsymbol{\sigma},A,\ell)=\left\{  j\in A:\sigma(A,j)=\ell\right\}
\]
and put
\[
r(\boldsymbol{\sigma},A,\ell)=\left\{
\begin{array}
[c]{ll}%
|S(\boldsymbol{\sigma},A,\ell)|-1,\text{ if }S(\boldsymbol{\sigma},A,\ell
)\neq\emptyset, & \\
0,\text{ otherwise.} &
\end{array}
\right.
\]
Now set
\[
\eta(\boldsymbol{\sigma},A)=\sum_{\ell\in\lbrack|A|]}r(\boldsymbol{\sigma
},A,\ell) ,
\]
\begin{equation}
\mathfrak{I}(\boldsymbol{\sigma})=\sum_{A\subseteq\lbrack m]:|A|\geq2}%
\eta(\boldsymbol{\sigma},A). \label{indice-induzione}%
\end{equation}
Notice that the index $\eta(\boldsymbol{\sigma},A)$ counts the number of ties among the elements of $A$ and that 
$\mathfrak{I}(\boldsymbol{\sigma})=0$ if and only if  $\boldsymbol{\sigma}\in\hat{\Sigma
}^{(m)}$.

\begin{theorem}
\label{forti-deboli} For any $m \in\mathbb{N}$ and for any $\boldsymbol{\sigma
}\in\Sigma^{(m)}$, there exists a voting situation $\mathcal{N}^{(m)}$ which
is $q$-concordant with $\boldsymbol{\sigma}$.
\end{theorem}

\begin{proof}
We partition the set $\Sigma^{(m)}$ as the union of the subsets
\[
\Sigma_{c}^{(m)}=\{\boldsymbol{\sigma}\in\Sigma^{(m)}:\mathfrak{I}
(\boldsymbol{\sigma})=\mathfrak{c}\}, \quad\mathfrak{c}=0,1,..., \sum_{\ell
=2}^{m} \binom{m}{\ell} (\ell-1)
\]
and the thesis will be proven by induction on the index $\mathfrak{c}$.

As to the first step of the induction we notice  that $\Sigma_{0}^{(m)}={\hat{\Sigma}}^{(m)}$ and that, for such a case,
the validity of the claim is then guaranteed by Theorem \ref{forti}.

As an inductive hypothesis we assume, for $k\in  \N \cup \{0\}$,  the existence of q-concordant voting
situations for any $\boldsymbol{\sigma}\in{\Sigma}_{\mathfrak{c}}^{(m)}$ with
$\mathfrak{c}\leq k $
and we aim to prove the existence of q-concordant voting situations for any
$\boldsymbol{\sigma}\in{\Sigma}_{k+1}^{(m)}$.

In what follows this task will be achieved by determining, for any such
$\boldsymbol{\sigma}$, two suitable ranking patterns $\boldsymbol{\sigma
^{\prime}},\,\boldsymbol{\sigma^{\prime\prime}}\in{\Sigma}_{k}^{(m)}$ and
respectively q-concordant voting situations $\mathcal{N}_{1}^{(m)}$,
$\mathcal{N}_{2}^{(m)}$.

For fixed $\boldsymbol{\sigma}\in{\Sigma}_{k+1}^{(m)}$ select a set
$A\subseteq\lbrack m]$ and an $\ell\in\lbrack|A|]$ such that
$S(\boldsymbol{\sigma},A,\ell)$ has cardinality larger or equal than two. The
existence of such a set $S(\boldsymbol{\sigma},A,\ell)$ is guaranteed by the
condition $\mathfrak{I}(\boldsymbol{\sigma})= k+1\geq1$. Fix furthermore $j\in
S(\boldsymbol{\sigma},A,\ell)$ and construct $\boldsymbol{\sigma^{\prime}}$
and $\boldsymbol{\sigma^{\prime\prime}}$ as follows. For $B\subseteq\lbrack
m]$ with $B\neq A$ set
\[
\sigma^{\prime}(B,h)=\sigma^{\prime\prime}(B,h)=\sigma(B,h),\quad\forall h\in
B.
\]
For any $h\in A$ such that $\sigma(A,h)\neq\ell$ we again require
\[
\sigma^{\prime}(A,h)=\sigma^{\prime\prime}(A,h)=\sigma(A,h).
\]
For any $i\in S(\boldsymbol{\sigma},A,\ell) \setminus\{j \}$ we set
\[
\sigma^{\prime}(A,i)=\ell,
\]
and
\[
\sigma^{\prime}(A,j)=\ell+|S(\boldsymbol{\sigma},A,\ell)| -1.
\]
Similarly, for any $i\in S(\boldsymbol{\sigma},A,\ell) \setminus\{j \}$, we
set
\[
\sigma^{\prime\prime}(A,i)=\ell+1,
\]
and
\[
\sigma^{\prime\prime}(A,j)=\ell.
\]
Notice that such a construction guarantees the condition $\mathfrak{I}%
(\boldsymbol{\sigma^{\prime}})=\mathfrak{I}(\boldsymbol{\sigma^{\prime\prime}%
})=k$~\ and, in view of the inductive hypotheses, it is possible to find out
\ two voting situations\ $\mathcal{N}_{1}^{(m)}$ and $\mathcal{N}_{2}^{(m)}$
such that $\mathcal{N}_{1}^{(m)}$ is q-concordant with $\boldsymbol{\sigma
^{\prime}}$ and $\mathcal{N}_{2}^{(m)}$ is q-concordant with
$\boldsymbol{\sigma^{\prime\prime}}$.

Denote by $n_{i}^{\prime}(A),n_{i}^{\prime\prime}(A)$ the number of votes
obtained by $i\in A$, respectively under $\mathcal{N}_{1}^{(m)}$ and
$\mathcal{N}_{2}^{(m)}$, when $A$ is the set of candidates. The above
construction guarantees the conditions $n_{i}^{\prime}(A)>n_{j}^{\prime}(A)$
and $n_{i}^{\prime\prime}(A)<n_{j}^{\prime\prime}(A)$. In their turn, such
conditions allow us, for fixed $i,j\in A$ with $\sigma(A,i)=\ell$ and $i\neq
j$, to consider the position
\begin{equation}
\label{ProducedVotSit}
{\mathcal{N}}^{(m)}=(n_{j}^{\prime\prime}(A)-n_{i}^{\prime\prime}%
(A))\times\mathcal{N}_{1}^{(m)}+(n_{i}^{\prime}(A)-n_{j}^{\prime}%
(A))\times\mathcal{N}_{2}^{(m)},
\end{equation}
which defines a bona-fide voting situation. It remains to prove that
${\mathcal{N}}^{(m)}$ is q-concordant with $\boldsymbol{\sigma}$, as wanted.

As in the proof of  Proposition \ref{trasforma} and by Proposition \ref{prop+}, the number of votes obtained by $i$, under
$\mathcal{N}^{(m)}$, when $A$ is the set of candidates, is given by
\[
{n}_{i}(A)=(n_{j}^{\prime\prime}(A)-n_{i}^{\prime\prime}(A))\cdot
n_{i}^{\prime}(A)+(n_{i}^{\prime}(A)-n_{j}^{\prime}(A))\cdot n^{\prime\prime
}_{i}(A)=n_{j}^{\prime\prime}(A)n_{i}^{\prime}(A)-n_{j}^{\prime}%
(A)n_{i}^{\prime\prime}(A);
\]
analogously
\[
{n}_{j}(A)=(n_{j}^{\prime\prime}(A)-n_{i}^{\prime\prime}(A))\cdot
n_{j}^{\prime}(A)+(n_{i}^{\prime}(A)-n_{j}^{\prime}(A))\cdot n_{j}%
^{\prime\prime}(A)=-n_{j}^{\prime}(A)n_{i}^{\prime\prime}(A)+n_{j}%
^{\prime\prime}(A)n_{i}^{\prime}(A).
\]
Therefore all the candidates belonging to $S(\boldsymbol{\sigma},A,\ell)$
obtain, according to ${\mathcal{N}}^{(m)}$, the same number of votes when the
set of candidates is $A$.

By construction of $\mathcal{N}^{(m)}$, all the other relations of equality or
inequality within pairs of candidates, respectively coincide with those
induced by $\mathcal{N} _{1}^{(m)}$ and $\mathcal{N}_{2}^{(m)}$. Therefore
$\mathcal{N}^{(m)}$ is q-concordant with $\boldsymbol{\sigma}$.
\end{proof}

\section{On the set of the values  $n(\mathcal{N}^{(m)})$ for q-concordant voting situations}\label{sec4}
A classical issue in voting theory is the analysis of the number of voters required for obtaining voting situations which may give rise to selected types of voting paradoxes (see in particular \cite{ErdosMoser}, \cite{Stearns}). In this section we point out some general aspects concerning voting situations q-concordant with ranking patterns.
Still in this analysis, it is convenient to separate between the two cases of strict or weak ranking patterns. The latter case will be considered first and a related result will be based on Theorem \ref{forti-deboli}. Later on we turn to considering the case of strict ranking patterns, where one can rely on Theorem 2 of \cite{DS22} (see also the Appendix). As it may be expected, more precise and stronger results will be obtained in such a case.
In order to formulate our results we need the following notation.
Let $m \in \N $ be fixed. I.e. fix the set $[m]$ of possible candidates and define the sets
$\Theta_m $ and $ \hat \Theta_m$ as follows.
\begin{equation}
\label{set-sol}
\Theta_m = \{ \theta \in \N : \forall  \boldsymbol{\sigma}  \in \Sigma^{(m)} \quad \exists  \mathcal{N}^{(m)} \text{ q-concordant with }  \boldsymbol{\sigma} \text{ with } n ( \mathcal{N}^{(m)}  )= \theta   \} ,
\end{equation}
analogously
\begin{equation}
\label{set-sol2}
\hat \Theta_m = \{ \theta \in \N : \forall  \boldsymbol{\sigma}  \in \hat  \Sigma^{(m)} \quad \exists  \mathcal{N}^{(m)} \text{ q-concordant with }  \boldsymbol{\sigma} \text{ with } n ( \mathcal{N}^{(m)}  )= \theta   \} .
\end{equation}
Let $\boldsymbol{\sigma}$ $\in\Sigma^{\left(  m\right)  }$ be an arbitrary
ranking pattern and assume the presence of $ \theta$ voters, with $\theta$
belonging to $\Theta_{m}$. It is assured, in this case, that the different
linear preference rankings (i.e. the different elements of $\Pi_{m}$) can be
distributed in such a way to produce a voting situation q-concordant with
$\boldsymbol{\sigma}$. One can then wonder if such an integer $\theta$ can really exist. We are
going to check next that $\Theta_{m}$ is a non-empty set, containing infinite
elements. It is furthermore obvious from the above definition that $\Theta
_{m}\subseteq\hat{\Theta}_{m}$ and thus $\hat{\Theta}_{m}$ is a non-empty set,
containing infinite elements, as well.

\begin{theorem}\label{infsol}
For any $m \in \N$ the following claims hold
\begin{itemize}
\item[i)] $ \Theta_m$ is non-empty;
\item[ii)]  any $\theta \in \Theta_m$ is a multiple of
$\lcm \{   2, \ldots , m \}$;
\item[iii)]   if $\theta \in \Theta_m$ then $(\theta + m !) \in \Theta_m $.
\end{itemize}
\end{theorem}

\begin{proof} We start  by proving  i). By Theorem \ref{forti-deboli}, for
$\boldsymbol{\sigma} \in \Sigma^{(m)}$ one can select a voting situation
$ \mathcal{N}_{\boldsymbol{\sigma}}^{(m)}  $ q-concordant with $\boldsymbol{\sigma}$ and, by recalling the position
(\ref{TotalNumberofVoters}),  set $n_{\boldsymbol{\sigma}}:=n\left(
\mathcal{N}_{\boldsymbol{\sigma}}^{(m)}\right)  $. Set furthermore
$$
a (m ) := \lcm \{       n_{\boldsymbol{\sigma} }   :  \boldsymbol{\sigma} \in \Sigma^{(m)}  \}.
$$
By Proposition \ref{trasforma}, the voting situation $\frac{a\left(  m\right)
}{n_{\boldsymbol{\sigma}}}\times\mathcal{N}_{\boldsymbol{\sigma}}^{(m)}$ is
q-concordant with $\boldsymbol{\sigma}$ \ as well. Moreover,
\[
n\left(  \frac{a(  m)  }{n_{\boldsymbol{\sigma}}}\times
\mathcal{N}_{\boldsymbol{\sigma}}^{(m)}\right)  =\frac{a(  m)
}{n_{\boldsymbol{\sigma}}}\cdot n\left(  \mathcal{N}_{\boldsymbol{\sigma}%
}^{(m)}\right)  =a(  m)  .
\]
Thus $a (m) \in\Theta_{m}$ and $\Theta_{m}$ is non-empty.

Proof of ii). Let $A \subseteq [m ] $ with $|A|=k \in \{2, \ldots , m\}$ and consider  the special ranking pattern $\hat{ \boldsymbol{\sigma} }\in \Sigma^{(m)}$ such that $\hat{ \sigma} (A, i) =1 $ for any $i \in A$. Therefore, for any voting situation $ \mathcal{N}_{\hat{ \boldsymbol{\sigma}}}^{(m)} $  q-concordant with $\hat{ \boldsymbol{\sigma}}$, the integer  $n_{\hat{\boldsymbol{\sigma}} } $  should be a multiple of $k$. Thus all the elements of
$\Theta_m$ are multiples of $\lcm \{ 2, \ldots , m \}$.

Proof of iii). In view of i) we can fix $\theta\in\Theta_{m}$. For
$\boldsymbol{\sigma}\in\Sigma^{(m)}$, let  $\mathcal{N}_{\boldsymbol{\sigma}%
}^{(m)}$ be a voting situation q-concordant with $\boldsymbol{\sigma}$  and
such that $n_{\boldsymbol{\sigma}} =\theta$.
By Proposition \ref{trasforma}, $1$\textcircled{+}$\mathcal{N}%
_{\boldsymbol{\sigma}}^{(m)}$ is q-concordant with $\boldsymbol{\sigma}$, as
well. The proof can thus be concluded by noticing that $n(1$%
\textcircled{+}$\mathcal{N}_{\boldsymbol{\sigma}}^{(m)})=m!+n(\mathcal{N}%
_{\boldsymbol{\sigma}}^{(m)})=m!+\theta$.
\end{proof}
As a consequence of items  i) and iii) of Theorem \ref{infsol} it is seen that, for whatever $m \in \N $, the set $\Theta_m$ contains infinite natural numbers.

\medskip
We now turn to specifically consider the set $ \hat \Theta_m  \supset \Theta_m$. In this case it is possible to explicitly determine an element $\bar \theta_m \in  \hat \Theta_m $ and, furthermore, to check that all the integers large enough
are eventually contained in $ \hat \Theta_m $.
On this purpose it necessary, once again, to rely on Theorem 2  of \cite{DS22} and on the arguments contained in Appendix and related with the load-sharing models of the special type $LS(\boldsymbol{\varepsilon},\boldsymbol{\sigma})$. Set
\begin{equation}
\bar{\theta}_{m}:=\prod_{h=2}^{m}[h\cdot z_{m}^{h-1}-\frac{h(h-1)}%
{2}],\label{DefinThetaSegnato}%
\end{equation}
where
\begin{equation}
z_{m}:=17\cdot m\cdot m! . 
\end{equation}
We notice that the integer $z_{m}$ is such that, for the quantities
$\varepsilon(2),...,\varepsilon(m)$ introduced in \eqref{stimaepsilon2} one has
$\varepsilon(h)=z_{m}^{-h+1}$.
\begin{theorem}
\label{infsol2} For any $m \in\N$ one has

\begin{itemize}
\item[i)] the integer $\bar{\theta}_{m}$ belongs to $\hat{\Theta}%
_{m}$;

\item[ii)] each integer $\theta\geq\bar{\theta}_{m}\cdot m!$ belongs
to $\hat{\Theta}_{m}$.
\end{itemize}
\end{theorem}
\begin{proof}
By Theorem 2  of \cite{DS22} and Corollary \ref{MainCrollary},
for a given $\boldsymbol{\sigma}\in \hat \Sigma^{(m)}$ the $LS(\boldsymbol{\varepsilon},\boldsymbol{\sigma})$ model with
$$
\boldsymbol{\varepsilon} =(0, z_m^{-1},z_m^{-2}, \ldots , z_m^{-m+1} )
$$
is p-concordant with $\boldsymbol{\sigma}$. By \eqref{dispoNv}, \eqref{scelta1Nv} and \eqref{MdefNv},
we  can write an explicit formula for  the probabilities $p_m (j_1, \ldots , j_m)$'s related to such a model. For any $(j_1, \ldots , j_m) \in \Pi_m $,
$$
p_m (j_1, \ldots , j_m)= \frac{1-(\sigma([m],j_1)  -1)z_m^{1-m} }{\hat M_0}
\cdot  \frac{1-(\sigma([m] \setminus \{j_1 \},j_2)  -1)z_m^{2-m} }{\hat M_1} \cdot
$$
\begin{equation}\label{zenaide}
\cdot  \frac{1-(\sigma([m] \setminus \{j_1 , j_2\},j_3)  -1)z_m^{3-m} }{\hat M_2}
\cdots  \frac{1-( \sigma([m]\setminus \{j_1, j_2, \ldots , j_{m-2}\},j_{m-1})-1 )z_m^{-1}}{\hat  M_{m-2}}.
\end{equation}
By \eqref{MdefNv}, the $k$-th term of the previous product  is
$$
\frac{1-(\sigma([m] \setminus \{j_1 , \ldots , j_{k-1}\},j_k)  -1)z_m^{k-m} }{\hat M_{k-1}} = \frac{1-(\sigma([m] \setminus \{j_1 , \ldots , j_{k-1}\},j_k)  -1)z_m^{k-m} }{m+1-k +\frac{(m+1-k)(m-k ) }{2}z_m^{k-m} }=
$$
\begin{equation}\label{rapporto}
=\frac{z_m^{m-k} -\sigma([m] \setminus \{j_1 , \ldots , j_{k-1}\},j_k)  +1}{
z_m^{m-k}(m+1-k) +\frac{(m+1-k)(m-k ) }{2} }.
\end{equation}
The numerator and denominator of \eqref{rapporto} are integers.
Thus, multiplying by the denominator 
\begin{equation}\label{pepepe}
\left (z_m^{m-k}(m+1-k) +\frac{(m+1-k)(m-k ) }{2} \right )
\end{equation}
the ration \eqref{rapporto}, one obtains an integer number.
Noticing  that
$$
  \prod_{k=1}^{m-1}      \left (z_m^{m-k}(m+1-k) +\frac{(m+1-k)(m-k ) }{2} \right )      =   \prod_{h=2}^{m} \left  [h z_m^{h-1}-
\frac{h(h-1)}{2}   \right ] ,
$$
the equation \eqref{zenaide} yields
\begin{equation}\label{qesp}
p_m (j_1, \ldots , j_m) \cdot \prod_{h=2}^{m} \left  [h z_m^{h-1}-
\frac{h(h-1)}{2}   \right ] =
p_m (j_1, \ldots , j_m) \cdot \bar \theta_m
\in \N, \text{ } (j_1, \ldots , j_m)  \in \Pi_m.
\end{equation}
The integer $\bar \theta_m := \prod_{h=2}^m[h z_m^{h-1}- \frac{h(h-1)}{2}] $ does not depend on $\boldsymbol{\sigma}$, hence $\bar \theta_m \in \hat\Theta_m$, by Proposition \ref{EquivConcorda}.

\medskip

We now prove ii). Let us label by $\mathbf{j}_1, \mathbf{j}_2,
\ldots , \mathbf{j}_{m!} $ the elements in $\Pi_m $.
 For 
$\boldsymbol{\sigma}\in \hat \Sigma^{(m)}$ we consider the voting situations $\mathcal{N}^{(m)}_{\boldsymbol{\sigma}, \bar \theta_m }$ q-concordant with $\boldsymbol{\sigma}$ and such that
$ n( \mathcal{N}^{(m)}_{\boldsymbol{\sigma}, \bar \theta_m }  ) =\bar \theta_m    $.
For $c \in \N$ 
we now consider the voting situation
$ \mathcal{N}^{(m)}_{\boldsymbol{\sigma},  c \cdot  \bar \theta_m } = c \times \mathcal{N}^{(m)}_{\boldsymbol{\sigma} , \bar  \theta_m} $.

Thus the following condition
\begin{equation}\label{condiz}
\forall A \subseteq [m], \text { with } i, \, j \in A \text{ and } i \neq j \text{ then } |n_i(A)-n_j (A)| \geq c
\end{equation}
is satisfied when the quantities $n_i(A)$ and $n_j (A)$ are those related to $ \mathcal{N}^{(m)}_{\boldsymbol{\sigma}, c  \cdot\bar   \theta_m }$.
For $\theta>\bar{\theta}_m \cdot m!$ we will construct a voting situation
$$
\mathcal{N}^{(m)}_{\boldsymbol{\sigma} , \theta } = \{N^{(m)}_{\boldsymbol{\sigma} , \theta }(\mathbf{j}_i)  : i \in [m!] \}
$$
q-concordant with $\boldsymbol{\sigma} $ and such that
$n (\mathcal{N}^{(m)}_{\boldsymbol{\sigma} , \theta }) = \theta$.
Let
\begin{equation}\label{numero-preferenze}
N^{(m)}_{\boldsymbol{\sigma} , \theta }(\mathbf{j}_i) = m! \, \cdot \,N^{(m)}_{\boldsymbol{\sigma} ,  \bar \theta_m }(\mathbf{j}_i) + \left \lfloor
 \frac{\theta  }{m ! }\right \rfloor - \bar \theta_m + \chi_i ,
\end{equation}
where
$$
\chi_i = \left \{
\begin{array}{ll}
1, \text{ if }  [\theta  \mod(m!) ]- i \geq 0, \\
0, \text{ otherwise.}
\end{array}
 \right .
$$
Let us check that the cardinality of $\mathcal{N}^{(m)}_{\boldsymbol{\sigma} , \theta } $ is $\theta$:
$$
n (\mathcal{N}^{(m)}_{\boldsymbol{\sigma} , \theta })= \sum_{i = 1}^{m!} N^{(m)}_{\boldsymbol{\sigma} , \theta }(\mathbf{j}_i) = \sum_{i = 1}^{m!} \left [ m! \, \cdot
 \,N^{(m)}_{\boldsymbol{\sigma} ,  \bar \theta_m }(\mathbf{j}_i) + \left \lfloor
 \frac{\theta  }{m ! }\right \rfloor - \bar \theta_m + \chi_i \right ] =
$$
$$
=m !\cdot  n (\mathcal{N}^{(m)}_{\boldsymbol{\sigma} , \bar \theta_m }) +  \left \lfloor
 \frac{\theta  }{m ! }\right \rfloor \cdot m!  - \bar  \theta_m \cdot m!
+ \sum_{i=1}^{m!} \chi_i =
$$
$$
=m ! \cdot \bar \theta_m +  \left \lfloor
 \frac{\theta  }{m ! }\right \rfloor \cdot m!  - \bar  \theta_m \cdot m! + [\theta  \mod(m!) ] = \theta.
$$
Now fix $\theta$ as  a multiple  of $m!$ and such that $ \theta \geq  \bar \theta_m \cdot m!$. In such a case  the
voting situation  $\mathcal{N}^{(m)}_{\boldsymbol{\sigma} , \theta }$  (defined
by \eqref{numero-preferenze}) coincides with
$$
\left (\frac{\theta}{m!} - \bar \theta_m \right )\textcircled{+} \left [ ( m! ) \times  \mathcal{N}^{(m)}_{\boldsymbol{\sigma} , \bar \theta_m }\right  ].
$$
Thus, by Proposition \ref{trasforma},  $\mathcal{N}^{(m)}_{\boldsymbol{\sigma} , \theta }$ is q-concordant with $\boldsymbol{\sigma}$ as well.

We now consider a $\theta' > \bar \theta_m \cdot m ! $ which is not a multiple of $m!$. We consider 
$$
\theta = \left \lfloor
 \frac{\theta'  }{m ! }\right \rfloor   m! 
$$
The condition \eqref{condiz}, with $c = m!$, is satisfied
with $n_i(A)$ and $n_j (A)$ that are quantities related to $\mathcal{N}^{(m)}_{\boldsymbol{\sigma} , \theta }$. In order to conclude the proof it is enough to notice that, being $\sum_{i=1}^{m!} \chi_i \leq  m ! -1$,   all the inequalities  between  $n_i(A)$ and $n_j (A)$ are maintained for  the voting situation $\mathcal{N}^{(m)}_{\boldsymbol{\sigma} , \theta' }$.
\end{proof}

 For a strict ranking pattern $\boldsymbol{\sigma}\in\hat{\Sigma}^{(m)}$ not
only we know that voting situations exist which are q-concordant and are based
on $\bar{\theta}_{m}$ voters, as guaranteed by the preceding result. Furthermore, we can
also produce explicitly one of such voting situations by letting
\[
N^{(m)}(j_{1},\ldots,j_{m}):=\left[  z_{m}^{m-1}-\sigma([m],j_{1})+1\right]
\times\left[  z_{m}^{m-2}-\sigma([m]\setminus\{j_{1}\},j_{2})+1\right]  \times
\]%
\begin{equation}
\times\left[  z_{m}^{m-3}-\sigma([m]\setminus\{j_{1},j_{2}\},j_{3})+1\right]
\times\cdots\times\left[  z_{m}-\sigma([m]\setminus\{j_{1},j_{2}%
,\ldots,j_{m-2}\},j_{m-1})+1\right]  .\label{SpecialVotSit}%
\end{equation}
 As a matter of fact,  the next result will be proven by just rephrasing arguments contained in the proof of Theorem \ref{infsol2}.
\begin{theorem}
\label{vote} For any integer $m\geq2$, for any $\boldsymbol{\sigma}\in
\hat{\Sigma}^{(m)}$, the voting situation $\mathcal{N}_{\boldsymbol{\sigma
},\bar{\theta}}^{(m)}:=\{N^{(m)}(j_{1},\ldots,j_{m}):(j_{1},\ldots,j_{m}%
)\in\Pi_{m}\}$ with $N^{(m)}(j_{1},\ldots,j_{m})$ defined by
\eqref{SpecialVotSit}   is q-concordant with $\boldsymbol{\sigma}$ and
$n(\mathcal{N}_{\boldsymbol{\sigma},\bar{\theta}_{m}}^{(m)})=\bar{\theta}_{m}$.
\end{theorem}

\begin{proof}
For any  $(j_{1},\ldots,j_{m})\in\Pi_{m}$ the quantity $N^{(m)}(j_{1}%
,\ldots,j_{m})$ in \eqref{SpecialVotSit} coincides with the product $p_m (j_1, \ldots , j_m) \cdot \bar \theta_m$ in \eqref{qesp}. By Theorem \ref{EpsilonEspliciti} in the Appendix and by
\eqref{qspiega} we know that the model $LS(\boldsymbol{\varepsilon
},\boldsymbol{\sigma})$ is p-concordant with $\boldsymbol{\sigma}$. By
applying Proposition \ref{EquivConcorda} we obtain that $\mathcal{N}%
_{\boldsymbol{\sigma},\bar{\theta}_{m}}^{(m)}$ is q-concordant with
$\boldsymbol{\sigma}$. The claim $n(\mathcal{N}_{\boldsymbol{\sigma}%
,\bar{\theta}_{m}}^{(m)})=\bar{\theta}_{m}$ immediately follows by \eqref{qesp}.
\end{proof}

\section{Discussion and concluding remarks}\label{sec5}
In this paper we have considered a standard context of voting theory based on a
family of possible elections with a fixed set of voters and a fixed set of
potential candidates. The different elections are respectively
characterized by different subsets of participating candidates and, looking
at the differences among their respective outcomes, attention has been focused on the concept of
voting situation q-concordant with a ranking pattern.
For such voting
situations we have obtained results concerning with the issues of existence and
construction and with cardinalities of the related
voters' populations. Such results aim to show that it is possible
to observe any arbitrary set of elections' outcomes and they thus lead to conclusions in the same direction of
classical results by Saari (see e.g. \cite{Saari(dictionary)}).

For our developments we have used results based on the concept of
ranking pattern $p$-concordant with probability models for non-negative random
variables and on a special class of load-sharing models (as presented in
\cite{DS22}).

In our analysis we have distinguished between the two cases of weak or strict
 ranking patterns, aiming to highlight the differences between the two cases.
 Obviously, the cardinality $|\hat{\Sigma}^{(m)}|$ of the set of all the strict
ranking patterns is smaller than $|\Sigma^{(m)}|$. In particular we point
out the following relations:
\[
\prod_{h=2}^{m}(h!)^{\binom{m}{h}}=|\hat{\Sigma}^{(m)}|<|\Sigma^{(m)}|<\prod_{h=2}^{m}h^{h\binom{m}{h}}.
\]
 where the last inequality can be explained by noticing that the term $\binom{m}{h}$ denotes the
number of subsets of $[m]$ with cardinality $h$ and $h^{h}$ is the total
number of functions $\sigma(A,\cdot):A\rightarrow\{1,2,\ldots,|A|\}$, which
are not necessarily ranking functions.

Constructive and quantitative results have been given in a closed form in the case of strict ranking patterns, 
whereas the results for the weak case are not in a closed form. 
 In this respect, however, a main implication of our work
concerns with the existence of probability models $p$-concordant with ranking
patterns. In fact, the passage to the setting of voting theory allows us to
extend an existence result, given in \cite{DS22}, from the strict case to the
weak case. 
More in details this issue will be discussed in the next Remark
\ref{ExtensionToWeak}. 
Such an extension may have interesting applications also to the analysis of paradoxes arising in different fields of probability such as 
the context of intransitive dice, the classic
games based  on the occurrence of competing
events in a sequence of trials, times of first occurrence for different words
 in random sampling of letters from an alphabet (see e.g. \cite{D21ranking}, \cite{DSS2}, \cite{GO1981}).

\begin{rem}
\label{ExtensionToWeak}
For given $m\geq2$, let $\boldsymbol{\sigma}\in
\Sigma^{(m)}$ be a given weak ranking pattern. By means of Theorems
\ref{forti} and \ref{forti-deboli}, existence has been proven for a voting
situation $\mathcal{N}_{\boldsymbol{\sigma}}^{(m)}\equiv$
$\{N_{\boldsymbol{\sigma}}^{(m)}(j_{1},\ldots,j_{m}):(j_{1},\ldots,j_{m}%
)\in\Pi_{m}\}$ q-concordant with $\boldsymbol{\sigma}$.\ In view of
Proposition \ref{EquivConcorda}, we know that a probability model p-concordant
with $\boldsymbol{\sigma}$  is one satisfying the condition \eqref{q=pBis}.

The existence of a
load-sharing model satisfying such a condition, on the other hand, is
guaranteed by Theorem 1 in the previous article \cite{DS22}. We can thus
conclude that, for any $\boldsymbol{\sigma}\in\Sigma^{(m)}$, there exists a (order dependent) 
load-sharing model \ p-concordant with $\boldsymbol{\sigma}$. Such a
conclusion then fills a gap, present in \cite{DS22}, where the existence of a
survival model p-concordant with $\boldsymbol{\sigma}$ was only guaranteed for
the case of a strict ranking pattern $\boldsymbol{\sigma}\in\hat{\Sigma}%
^{(m)}$. The following difference between the two cases emerges however. In
the case $\boldsymbol{\sigma}\in\hat{\Sigma}^{(m)}$ the existence is proven of
a non-order dependent load sharing model  p-concordant with
$\boldsymbol{\sigma}$ and it can be identified in terms of a closed formula.
In the case $\boldsymbol{\sigma}\in\Sigma^{(m)} \setminus \hat{\Sigma}^{(m)}$ we must
rely on Theorem 1 in \cite{DS22} which guarantees the existence of a
p-concordant  load sharing model that will generally be order dependent. Furthermore, the voting situation 
 $\mathcal{N}_{\boldsymbol{\sigma}}^{(m)} $ is
determined, in Theorem \ref{forti-deboli} on the basis of an inductive procedure, so that the related
construction is less explicit.
\end{rem}

Now we  present a few examples  in order to illustrate other aspects of the above results.
 The procedure for constructing a $q$-concordant voting situation starting from a strict ranking pattern is demonstrated in the next example. For the resulting voting situation, the corresponding number of voters is typically very large. One can however carry out some procedure to pass to a different $q$-concordant voting situation, in which the number of voters can be drastically reduced. On this purpose the operations on the voting situations, as defined in Section \ref{section 3}, can be conveniently used.
\begin{example}
\label{extra} For $m=3$ consider the strict ranking pattern $\boldsymbol{\sigma}$ defined as follows:
\[
\sigma([3],1)=1,\sigma([3],2)=2,\sigma([3],3)=3,
\]%
\[
\sigma(\{1,2\},1)=2,\sigma(\{1,2\},2)=1,
\]%
\[
\sigma(\{1,3\},1)=2,\sigma(\{1,3\},3)=1,
\]%
\[
\sigma(\{2,3\},2)=2,\sigma(\{2,3\},3)=1.
\]
In view of Theorems \ref{th1semplificato} and \ref{EpsilonEspliciti} in the Appendix p-concordant load-sharing models can be singled out by suitably
choosing coefficients $\varepsilon(2),\varepsilon(3)$ and by consequently imposing parameters of the following form:
\[
\mu_{1}(\emptyset)=1,\text{ }\mu_{2}(\emptyset)=1-\varepsilon(3),\text{ }%
\mu_{3}(\emptyset)=1-2\varepsilon(3)
\]%
\[
\mu_{1}(\{2\})=1-\varepsilon(2),\text{ }\mu_{3}(\{2\})=1,\mu_{2}%
(\{1\})=1-\varepsilon(2),
\]%
\[
\mu_{3}(\{1\})=1,\text{ }\mu_{1}(\{3\})=1-\varepsilon(2),\mu_{2}(\{3\})=1.
\]
By taking the special choice \eqref{stimaepsilon2}, one obtains a load sharing
model with
\[
\varepsilon(2)=\frac{1}{17\times9\times2}=\frac{1}{306}{\small ,\ }%
\varepsilon(3)=\varepsilon(2)^{2}=\frac{1}{306^{2}},
\]
that it is seen to be p-concordant by using Corollary \ref{MainCrollary} in
the Appendix. For brevity sake, set $a=\frac{1}{306}$. In view of formula
\eqref{MdefNv} one obtains that the corresponding set function
$M$, defined in \eqref{muodeNv}, is given by $\widehat{M}_{3}=3(1-a^{2}),$ $\widehat{M}_{2}=2-a$. By
formula \eqref{dispoNv}, concerning the probabilities of permutations, we
obtain
\[
p(1,2,3)=\frac{1}{3(1+a)(2-a)},\text{ }p(1,3,2)=\frac{1}{3(1-a)^{2}(2-a)},
\]%
\[
p(2,1,3)=\frac{1-a}{3(2-a)},\text{ }p(2,3,1)=\frac{1}{3(2-a)},
\]%
\[
p(3,1,2)=\frac{1-2a^{2}}{3(1+a)(2-a)},\text{ }p(3,2,1)=\frac{1-2a^{2}%
}{3(1-a)^{2}(2-a)}.
\]
By multiplying all the above probabilities by $3(1-a)^{2}(2-a)$ we obtain
polynomials of degrees at most equal to three. Consequently a set of all
natural numbers is obtained by multiplying by $(306)^{3}$ in order to
determine a q-concordant voting situation $\widehat{\mathcal{N}}$. Starting
from $\widehat{\mathcal{N}}$ and suitably applying the operations of
Multiplication, Addition, and Internal addition, as defined in Section \ref{section 3},  we
can progressively obtain different voting situations admitting smaller and
smaller number of voters, but still q-concordant with $\boldsymbol{\sigma}$.
At the end we can obtain the following voting situation where the number of
voters is $43$:
\[
N(1,2,3)=2,\text{ }N(1,3,2)=15, N(2,1,3)=1,
\]%
\[
N(2,3,1)=13,\text{ }N(3,1,2)=0,N(3,2,1)=12.
\]

\end{example}

\medskip
In the next example we employ the method, developed in the proof of Theorem
\ref{forti-deboli}, to build a voting situation q-concordant with a weak
ranking pattern by applying suitable operations over voting situations
q-concordant with strict ranking patterns.
\begin{example}
Let $m=3$ and consider the weak ranking pattern $\boldsymbol{\tilde{\sigma}}$
defined as follows:%
\[
\tilde{\sigma}([3],i)=i,\text{ for }i\in\lbrack3]
\]%
\[
\tilde{\sigma}(\{1,2\},1)=1,\quad\tilde{\sigma}(\{1,2\},2)=1,
\]%
\[
\tilde{\sigma}(\{1,3\},1)=2,\quad\tilde{\sigma}(\{1,3\},3)=1,
\]%
\[
\tilde{\sigma}(\{2,3\},2)=2,\quad\tilde{\sigma}(\{2,3\},3)=1.
\]
Consider furthermore the two strict ranking pattern $\boldsymbol{\sigma
}^{\prime}$ and $\boldsymbol{\sigma}^{\prime\prime}$ defined by
\[
\sigma^{\prime}([3],i)=i,\text{ for }i\in\lbrack3]
\]%
\[
\sigma^{\prime}(\{1,2\},1)=2,\quad\sigma^{\prime}(\{1,2\},2)=1,
\]%
\[
\sigma^{\prime}(\{1,3\},1)=2,\quad\sigma^{\prime}(\{1,3\},3)=1,
\]%
\[
\sigma^{\prime}(\{2,3\},2)=2,\quad\sigma^{\prime}(\{2,3\},3)=1,
\]
and
\[
\sigma^{\prime\prime}([3],i)=i,\text{ for }i\in\lbrack3]
\]%
\[
\sigma^{\prime\prime}(\{1,2\},1)=1,\quad\sigma^{\prime\prime}(\{1,2\},2)=2,
\]%
\[
\sigma^{\prime\prime}(\{1,3\},1)=2,\quad\sigma^{\prime\prime}(\{1,3\},3)=1,
\]%
\[
\sigma^{\prime\prime}(\{2,3\},2)=2,\quad\sigma^{\prime\prime}(\{2,3\},3)=1.
\]
A voting situation $\mathcal{N}'$ q-concordant with $\boldsymbol{\sigma}^{\prime}$ is
\[
N^{\prime}(1,2,3)=0,N^{\prime}(1,3,2)=4,N^{\prime}(2,1,3)=0,N^{\prime
}(2,3,1)=3,N^{\prime}(3,1,2)=0,N^{\prime}(3,2,1)=2.
\]
A voting situation $\mathcal{N}''$ q-concordant  with $\boldsymbol{\sigma}^{\prime\prime}$ is
\[
N^{\prime\prime}(1,2,3)=0,N^{\prime\prime}(1,3,2)=4,N^{\prime\prime
}(2,1,3)=0,N^{\prime\prime}(2,3,1)=3,N^{\prime\prime}(3,1,2)=1,N^{\prime
\prime}(3,2,1)=1.
\]
A voting situation $\tilde{\mathcal{N}}$ q-concordant with $\boldsymbol{\tilde{\sigma}}$
 can be obtained by operating over the afore-mentioned voting situations
$\mathcal{N}'$ and $\mathcal{N}''$, according to the method developed in the
proof of Theorem \ref{forti-deboli}.
More precisely, we apply the formula \eqref{ProducedVotSit} to the special case $A=\{1,2\},j=1,i=2$. Noticing that
\[
n_{1}^{\prime}(\{1,2\})=4,n_{2}^{\prime}(\{1,2\})=5,n_{1}^{\prime\prime
}(\{1,2\})=5,n_{2}^{\prime\prime}(\{1,2\})=4,
\]
from formula \eqref{ProducedVotSit}   one obtains
\[
\tilde{N}(j_{1},j_{2},j_{3})=N^{\prime}(j_{1},j_{2},j_{3})+N^{\prime\prime
}(j_{1},j_{2},j_{3}),\quad(j_{1},j_{2},j_{3})\in\Pi_{m},
\]
whence
\[
\tilde{N}(1,2,3)=0,\tilde{N}(1,3,2)=8,\tilde{N}(2,1,3)=0,\tilde{N}%
(2,3,1)=6,\tilde{N}(3,1,2)=1,\tilde{N}(3,2,1)=3.
\]
\end{example}

\medskip
In the following example we construct a $q$-concordant voting situation in correspondence with any   ranking pattern belonging to a special class. On this purpose we follow a logic similar to the one underlying the first step of the proof of Theorem \ref{th1semplificato} in the Appendix.
\begin{example}
Here we consider the class of the weak ranking patterns $\boldsymbol{\sigma}$
$\ $such that, for any $A\subseteq\lbrack m]$ with $|A|\geq3$,
\[
\sigma(A,i)=1,\quad\forall i\in A
\]
and, for the rest, presenting arbitrary behavior on the subsets $A$ with
$|A|=2$.

In other words, this class of ranking patterns gives rise to an extreme case
where arbitrary outcomes are admitted as far as elections with exactly two
candidates are considered, whereas all the candidates are perfectly equivalent
in all the elections with more than two candidates. For any such
$\boldsymbol{\sigma}$ \ we aim to determine a q-concordant voting situation.
Notice that the solution of this problem might be seen as a generalization of
the classical result by McGarvey \cite{McGarvey1953}. Such a solution can be determined by using a
procedure as the one appearing in the first step of the proof of Theorem
\ref{th1semplificato} and based on the formulae \eqref{UltimeEpsNulle} and
\eqref{agreements}. By following such a procedure one can obtain the voting
situation
\[
\mathcal{N}_{{}}^{(m)}=\{N (j_{1},\ldots,j_{m}
):(j_{1},\ldots,j_{m})\in\Pi_{m}\}
\]
with%

\begin{equation}
N (j_{1},\ldots,j_{m-1},j_{m})=\left\{
\begin{array}
[c]{ll}%
2, & \text{ if }\sigma(\{j_{m-1},j_{m}\},j_{m-1})=1\text{ and }\sigma
(\{j_{m-1},j_{m}\},j_{m})=2;\\
1, & \text{ if }\sigma(\{j_{m-1},j_{m}\},j_{m-1})=\sigma(\{j_{m-1}%
,j_{m}\},j_{m})=1;\\
0, & \text{ if }\sigma(\{j_{m-1},j_{m}\},j_{m-1})=2\text{ and }\sigma
(\{j_{m-1},j_{m}\},j_{m})=1.
\end{array}
\right.  \label{exTh}%
\end{equation}

Notice that the total number of voters here is $n(  \mathcal{N}
^{(m)})  =m!$. In particular, by taking $m=4$ and
\[
\sigma(\{1,2\},1)=1,\quad\sigma(\{1,2\},2)=2,\sigma(\{3,4\},3)=1,\quad
\sigma(\{3,4\},4)=2,
\]%
\[
\sigma(\{1,3\},1)=1,\quad\sigma(\{1,3\},3)=1,\sigma(\{2,3\},2)=1,\quad
\sigma(\{2,3\},3)=1,
\]%
\[
\sigma(\{1,4\},4)=1,\quad\sigma(\{1,4\},1)=2,\sigma(\{2,3\},2)=1,\quad
\sigma(\{2,3\},3)=2,
\]
we obtain a   q-concordant voting situation $ \mathcal{N}
^{(4)} $ such that   $n(  \mathcal{N}
^{(4)})  =24   $ and 
\[
N_{{}}(1,2,3,4)=N_{{}}(2,3,4,1)=N_{{}}(3,2,4,1)=N_{{}}(3,4,1,2)=N(2,1,3,4)=
\]%
\[
N(4,1,2,3)=N(1,4,2,3)=N(4,3,1,2)=2;
\]%
\[
N(1,3,2,4)=N(1,3,4,2)=N(3,1,2,4)=N(3,1,4,2)=
\]%
\[
=N(2,4,1,3)=N(2,4,3,1)=N(4,2,1,3)=N(4,2,3,1)=1;
\]%
\[
N(1,2,4,3)=N(2,3,1,4)=N(3,2,1,4)=N(2,1,4,3)=
\]%
\[
=N(1,4,3,2)=N(4,1,3,2)=N(3,4,2,1)=N(4,3,2,1)=0.
\]

\end{example}

\bibliographystyle{abbrv}


\newpage
\section{Appendix} \label{Appendix}  Attention will be concentrated on the case when the random
variables $X_{1},\dots,X_{m}$ admit an absolutely continuous joint probability
distribution. Besides the use of the corresponding joint density function,
such a joint distribution can be also described in terms of  the family of the
\textit{Multivariate Conditional Hazard Rate} (m.c.h.r.) functions, defined as follows:%

\[
\lambda_{j}(t|i_{1},\ldots,i_{k};t_{1},\ldots,t_{k}):=
\]%
\begin{equation}
\lim_{\Delta t\rightarrow0^{+}}\frac{1}{\Delta t}\mathbb{P}(X_{j}\leq t+\Delta
t|X_{i_{1}}=t_{1},\ldots,X_{i_{k}}=t_{k},X_{k+1:m}>t). \label{DefMCHR}%
\end{equation}
\begin{equation}
\lambda_{j}(t|\emptyset):=\lim_{\Delta t\rightarrow0^{+}}\frac{1}{\Delta
t}\mathbb{P}(X_{j}\leq t+\Delta t|X_{1:m}>t). \label{DefMCHR2}%
\end{equation}

For definitions and properties of m.c.h.r. functions see in particular
\cite{ShaSha}, the review paper \cite{ShaSha15} and, e.g., \cite{Spi19ASMBI} and
other references cited therein.

As pointed out in \cite{DMS20} and \cite{DS22}, the system of the m.c.h.r. functions
is convenient to analyze some aspects of the quantities $\alpha_{j}(A)$ defined  in
(\ref{alphasNv}).

 We focus on the special cases when  the $m$-tuple $\left(  X_{1},\ldots,X_{m}\right) $ is distributed according
to a \textit{order dependent load-sharing} model, i.e. when, for $k\in[ m-1]$, for distinct
$i_{1},\ldots,i_{k},j\in\lbrack m]$ and for an ordered sequence $0<t_{1}
<\cdots<t_{k}<t,$ one has
\begin{equation}
\lambda_{j}(t|i_{1},\ldots,i_{k};t_{1},\ldots,t_{k})=\mu_{j}(i_{1}
,\ldots,i_{k}),\lambda_{j}(t|\emptyset)=\mu_{j}(\emptyset),\label{DefTHLSMNv}
\end{equation}
for suitable non negative quantities $\mu_{j}(i_{1},\ldots,i_{k})$ and
$\mu_{j}(\emptyset)$.

It is in particular interesting the case when the functions
$\mu_{j}(i_{1},\ldots,i_{k})$ do not depend on the order of the components of
the vector $(i_{1},\ldots,i_{k})$. Such a case has been designated by the term
\emph{non-order dependent} load sharing and, with a minor abuse of notation,
sometime we write $\mu_{j}(I)$, with $I=\{i_{1},\ldots,i_{m}\}$, in place of
$\mu_{j}(i_{1},\ldots,i_{k})$.

For a fixed family $\mathcal{M}$ of parameters $\mu_{j}\left(  \emptyset
\right)  $ and $\mu_{j}(i_{1},\ldots,i_{k})$ as in (\ref{DefTHLSMNv}),\ for
$k\in\left[  m-1\right]  $ and for $i_{1}\neq\ldots\neq i_{k}$, set%

\begin{equation}
M(i_{1},\ldots,i_{k}):=\sum_{j\in\lbrack m]\setminus\{i_{1},\ldots,i_{k}\}}%
\mu_{j}(i_{1},\ldots,i_{k})\text{ and }M(\emptyset)=\sum_{j\in\lbrack m]}%
\mu_{j}(\emptyset). \label{muodeNv}%
\end{equation}

As a relevant property of the corresponding load sharing model, one has
$\mathbb{P}(J_{1}=j)=\frac{\mu_{j}(\emptyset)}{\text{ }M(\emptyset)}$ and
\begin{equation}
\mathbb{P}(J_{k+1}=j|J_{1}=i_{1},\,J_{2}=i_{2},\,\ldots,\,J_{k}=i_{k}%
)=\frac{\mu_{j}(i_{1},\ldots,i_{k})}{M(i_{1},\ldots,i_{k})}.
\label{ProbConditJinLSNv}%
\end{equation}
(see also \cite{Spi19ASMBI} and \cite{DMS20}). Concerning with the joint probability distribution of $\mathbf{J}\equiv\left(
J_{1},...,J_{m}\right) $ one immediately obtains the following consequence, for $h=2,...,m$:
\begin{equation}
\mathbb{P}(J_{1}=i_{1},\,J_{2}=i_{2},\,\ldots,\,J_{h}=i_{h})=\frac{\mu_{i_{1}%
}(\emptyset)}{M(\emptyset)}\frac{\mu_{i_{2}}(i_{1})}{M(i_{1})}\frac{\mu
_{i_{3}}(i_{1},i_{2})}{M(i_{1},i_{2})}\ldots\frac{\mu_{i_{k}}(i_{1}%
,i_{2},\ldots i_{h-1})}{M(i_{1},i_{2},\ldots i_{h-1})}. \label{dispoNv}%
\end{equation}
See also \cite{Spi19ASMBI}. Now we need to introduce the following notation.

For $B \subset[m]$ and $k = 1, \ldots, m -|B|$ let us define
\begin{equation}
\mathcal{D}(B,k):=\{(i_{1},\ldots,i_{k}):i_{1},\ldots,i_{k}\not \in B\text{
and }i_{1}\neq i_{2}\neq\ldots\neq i_{k}\}. \label{insiemeD}%
\end{equation}
When $k=m-|B|$, $\mathcal{D}(B,k)$ is then the set of all the permutations of
the elements of $B^{c}$. In particular the set $\mathcal{D}(\emptyset,m)$
becomes $\Pi_{m}$. 

Concerning with the probabilities $\alpha_{j}(A)$ in the case of a load-sharing models, the following equation is readily obtained by combining relation (\ref{dispoNv}) with Proposition 1 in \cite{DS22}: 
\[
\alpha_{j}(A)=\mathbb{P}(X_{j}=\min_{i\in A}X_{i})=\frac{\mu_{j}(\emptyset
)}{M(\emptyset)}+
\]%
\begin{equation}
+\sum_{k=1}^{m-\ell}{\sum_{(i_{1},\ldots,i_{k})\in\mathcal{D}( A,k)}}\frac
{\mu_{i_{1}}(\emptyset)}{M(\emptyset)}\frac{\mu_{i_{2}}(i_{1} )}{M(i_{1}%
)}\ldots\frac{\mu_{i_{k}}(i_{1},i_{2},\ldots i_{k-1})}{M(i_{1} ,i_{2},\ldots
i_{k-1})}\frac{\mu_{j}(i_{1},i_{2},\ldots i_{k})}{M(i_{1} ,i_{2},\ldots
i_{k})}. \label{formulata}%
\end{equation}

\bigskip
Let now $\boldsymbol{\sigma}\in\hat{\Sigma}^{(m)}$ be an assigned strict
ranking pattern and let $\varepsilon(2),...,\varepsilon(m)$ be positive
quantities such that $\ (\sigma(A,i)-1)\varepsilon(|A|)<1$ \ for all
$A\subseteq\left[  m\right]  $ and $i\in A$.

Starting from $\boldsymbol{\sigma}$, in \cite{DS22} a special class of
load-sharing models has been defined by imposing parameters of the following
form
\begin{equation}
{\mu}_{i}([m]\setminus A)=1-(\sigma(A,i)-1)\varepsilon(|A|),%
A\subseteq\left[  m\right]  ,\text{ with }|A|\geq2,\text{ }i\in
A.\label{scelta1Nv}%
\end{equation}
For $A=\{i\}$ we finally set $\mu_{i}([m]\setminus A)=1$ and 
$\varepsilon(1)=0$. The load-sharing model corresponding to such a choice of
parameters is designated by the symbol $LS\left(  \boldsymbol{\varepsilon
},\boldsymbol{\sigma}\right)  $.

The interest in such a special class is justified by the following existence result.

\begin{theorem}
\label{th1semplificato} For $m\in\mathbb{N}$, let $\boldsymbol{\sigma}\in
\hat{\Sigma}^{(m)}$ be a ranking pattern. Then there exist constants
$\varepsilon(2),...,\varepsilon(m)$ such that a $m$-tuple $(X_{1},\ldots
,X_{m})$ distributed according to the model $LS\left(  \boldsymbol{\varepsilon
},\boldsymbol{\sigma}\right)  $, where $\boldsymbol{\varepsilon}=\left(  0,
\varepsilon(2),...,\varepsilon(m)\right)  $, is $p$-concordant with
$\boldsymbol{\sigma}$.
\end{theorem}

\bigskip

As a matter of fact, the following quantitative result has rather been proven
in \cite{DS22}.

\begin{theorem}
\label{EpsilonEspliciti} For any $\boldsymbol{\sigma}\in\hat{\Sigma}$ and any
$\boldsymbol{\varepsilon}=(0,\varepsilon(2),\ldots,\varepsilon(m))$ such that,
for $\ell=2,...,m-1,$
\begin{equation}
\frac{(m-\ell)!(\ell-1)!}{2\cdot m!}\varepsilon(\ell)>8\ell\varepsilon(\ell+1)
\label{stimaepsilonNv}%
\end{equation}
the model $LS(\boldsymbol{\varepsilon},\boldsymbol{\sigma})$ is $p$-concordant
with $\boldsymbol{\sigma}$.
\end{theorem}

\bigskip

The inequalities in \eqref{stimaepsilonNv} can be in particular obtained by
letting, for $h=2,...,m$,
\begin{equation}
\varepsilon(h)=(17\cdot m\cdot m!)^{-h+1}.\label{stimaepsilon2}%
\end{equation}
For a given $\boldsymbol{\sigma}\in\hat{\Sigma}^{(m)}$, by Theorem
\ref{EpsilonEspliciti} and \eqref{stimaepsilon2} one can thus conclude with the following
\begin{corollary} \label{MainCrollary}
An $m$-tuple $(X_{1},\ldots,X_{m})$ distributed according to a $LS(\boldsymbol{\varepsilon},\boldsymbol{\sigma})$  load-sharing model 
with parameters of the form
\begin{equation}
\mu_{j}([m]    \setminus A)=1-\frac{\sigma(A,j)-1}{(17\cdot m\cdot m!)^{|A|-1}%
},j\in A.\label{qspiega}%
\end{equation}
is $p$-concordant with $\boldsymbol{\sigma}$.
\end{corollary}

\medskip
Theorem \ref{EpsilonEspliciti} and Corollary \ref{MainCrollary}  are quantitative results, that have been used in Section \ref{sec4}. We are also interested in Theorem \ref{th1semplificato}, that has been used in Section \ref{section 3}. For this reason an autonomous proof of the latter, independent from the one of Theorem \ref{EpsilonEspliciti}, will be given next.

Notice the following implication of the above formula \eqref{formulata} for a load
sharing model described by the family $\mathcal{M\equiv}\{\mu_{j}%
(I): I\subset\left[  m\right] , j \not \in I  \}$: for given $A\subseteq\left[  m\right]$,
the probabilities $\{\alpha_{j}(A):j\in A\}$ only depend on $\{\mu
_{j}(I):I\subseteq A^{c},j\not \in I\}$.

\bigskip\begin{rem}\label{richiamare}
Specifically concerning with the case of a $LS\left(
\boldsymbol{\varepsilon},\boldsymbol{\sigma}\right) $ model, we can realize
that, for given $B\subseteq\left[  m\right]  $ with $|B|=n\leq m$ and $j\in
B$, the probability $\alpha_{j}(B)$ only depends on $\varepsilon
(n),\varepsilon(n+1)...,\varepsilon(m)$ and on the functions $\sigma(D,\cdot)$
for $D\subseteq\left[  m\right]  $ with $D\supseteq B$.
\end{rem}

Notice furthermore that, for an arbitrary choice of $\boldsymbol{\varepsilon
},\boldsymbol{\sigma}$, the $LS\left(  \boldsymbol{\varepsilon}%
,\boldsymbol{\sigma}\right)  $ model has the set of numbers $\mu
_{j}([m]\setminus A)$ (for $j\in A$) that only depends on the cardinality of $A$. 
Thus, by \eqref{muodeNv} and \eqref{scelta1Nv}, one has
\begin{equation}
\hat{M}_{m-h}= M(i_{1}, \ldots, i_{m-h}) = \sum_{u=1}^{h}[1-(u-1)\varepsilon
(h)]=h-\frac{h(h-1)} {2}\varepsilon(h), \label{MdefNv}%
\end{equation}
for $h\in\lbrack m]$ and for any $(i_{1}, \ldots, i_{m-h}) \in\mathcal{D}%
(\emptyset, m-h)$. The formulas \eqref{scelta1Nv}, \eqref{qspiega} and
(\ref{MdefNv}) give rise to a specially convenient form for the probability in \eqref{dispoNv}

\bigskip

We are now in a position to present the proof of Theorem \ref{th1semplificato}.
\begin{proof}
Let us fix the given ranking pattern
\[
\boldsymbol{\sigma}=(\sigma(A,\cdot):A\subseteq\left[  m\right]  ) \in
\hat\Sigma^{(m)}.
\]
In what follows we will inductively identify, for $n=2,\ldots,m$, a sequence
of vectors $\boldsymbol{\varepsilon}^{\left(  n\right)  }\equiv\left(
\varepsilon^{(n)}(\ell))_{\ell=2,\ldots,m}\right)  $ and we will consider the
load-sharing models $LS\left(  \boldsymbol{\varepsilon}^{\left(  n\right)
},\boldsymbol{\sigma}\right)  $ with parameters $\mu_{j}^{( n) }(B)$
determined by (\ref{scelta1Nv}) through $\boldsymbol{\sigma}$ and the vectors
of coefficients\emph{ }$\boldsymbol{\varepsilon}^{\left(  n\right)  }$. In
correspondence with $LS\left(  \boldsymbol{\varepsilon}^{\left(  n\right)
},\boldsymbol{\sigma}\right)  $, denote furthermore by $\alpha_{i}^{(n)}(A)$
($i\in A$)) the related quantities as defined in \eqref{alphasNv}.

Along the construction of $\boldsymbol{\varepsilon}^{\left(  n\right)  }%
\equiv\left(  \varepsilon^{(n)}(\ell))_{\ell=2,\ldots,m}\right)  $,
$n=2,...,m-1$, we will in particular impose the conditions%
\begin{equation}
\varepsilon^{(n)}(n+1)=\cdots=\varepsilon^{(n)}(m)=0 \label{UltimeEpsNulle}%
\end{equation}

and notice that, in view of the formula \eqref{formulata} and Remark
\ref{richiamare}, they imply that
\begin{equation}
\alpha_{i}^{\left(  n\right)  }(A)=\frac{1}{|A|} \label{UltimeAlfsCostanti}%
\end{equation}
for any $A$ $\subseteq\left[  m\right]  $ with $|A|>n$.

For $n=2,\ldots,m$ consider now the sequence of claims $\mathcal{H}(n)$
defined as follows.
\[
\mathcal{H}(n):\text{ By maintaining the condition \eqref{UltimeEpsNulle} it
is possible to find positive quantities }%
\]

$\varepsilon^{(n)}(2),\ldots,\varepsilon^{(n)}(m)\text{ such that }$%

\begin{equation}
\sigma(A,i)<\sigma(A,j)\Rightarrow\alpha_{i}^{(n)}(A)>\alpha_{j}%
^{(n)}(A),\label{agreements}%
\end{equation}
for any $A$ with $|A|\leq n$. Notice that $\mathcal{H}(m)$ coincides with the thesis of the theorem.  
\begin{equation}
\sigma(A,i)<\sigma(A,j)\Rightarrow\alpha_{i}(A)>\alpha_{j}(A),\label{finoam}%
\end{equation}
for any $A\subseteq\lbrack m]$.

The claims $\mathcal{H}(n)$ ($n=2,\ldots,m$) will now be proven by induction
on $n$. This procedure will then lead us to obtain the claim $\mathcal{H}(m)$,
namely the thesis of the Theorem.

In the first induction step $n=2$ we consider the ranking functions
$\sigma(A,\cdot) $ with $A\subseteq[ m] $ and $|A|=2$.

Initially, we set $\varepsilon^{(2)}(2)=\frac{1}{2}$ and $\varepsilon
^{(2)}(\ell)=0$ for any $\ell\geq3$.

In this way, still by formula \eqref{formulata} and equation
\eqref{scelta1Nv}, one has that for $A=\{i,j\}$, with $i \neq j$, the
implication
\[
\sigma(A,i)<\sigma(A,j)\Rightarrow\alpha_{i}^{(2)}(A)>\alpha_{j}^{(2)}(A),
\]
is satisfied.

This claim immediately follows from the circumstance that $\varepsilon
^{(2)}(3)=\cdots=\varepsilon^{(2)}(m)=0$ implies that all the terms, appearing
in the summation in \eqref{formulata} are equal each
other, excepting
\[
\frac{\mu_{i_{1}}^{(2)}(\emptyset)}{M^{(2)}(\emptyset)}\frac{\mu_{i_{2}}%
^{(2)}(i_{1})}{M^{(2)}(i_{1})}\ldots\frac{\mu_{i_{k}}^{(2)}(i_{1},i_{2},\ldots
i_{m-3})}{M^{(2)}(i_{1},i_{2},\ldots i_{m-3})}\frac{\mu_{i}^{(2)}(i_{1}%
,i_{2},\ldots i_{m-2})}{M^{(2)}(i_{1},i_{2},\ldots i_{m-2})}\neq
\]%
\[
\frac{\mu_{i_{1}}^{(2)}(\emptyset)}{M^{(2)}(\emptyset)}\frac{\mu_{i_{2}}%
^{(2)}(i_{1})}{M^{(2)}(i_{1})}\ldots\frac{\mu_{i_{k}}^{(2)}(i_{1},i_{2},\ldots
i_{m-3})}{M^{(2)}(i_{1},i_{2},\ldots i_{m-3})}\frac{\mu_{j}^{(2)}(i_{1}%
,i_{2},\ldots i_{m-2})}{M^{(2)}(i_{1},i_{2},\ldots i_{m-2})}.
\]

We notice that the previous inequalities $\alpha_{i}^{(2)}(A)>\alpha_{j}
^{(2)}(A)$ are strict.

Let us now pass, on the other hand, to considering a set $A\subseteq[m]$ with
$|A|>2$. By \eqref{formulata} and equation \eqref{scelta1Nv} one has
\[
\alpha_{i}^{(2)}(A)=\frac{1}{|A|},
\]
regardless of the index $i\in A$.

We now proceed inductively. For given $n<m$ we maintain \eqref{UltimeEpsNulle}
and assume the induction hypothesis $\mathcal{H}(n)$, guaranteeing the
existence of $\varepsilon^{(n)}(2),\varepsilon^{(n)}(3),\ldots,\varepsilon
^{(n)}(n)$ such that the implication
\begin{equation}
\sigma(A,i)<\sigma(A,j)\Rightarrow\alpha_{i}^{(n)}(A)>\alpha_{j}^{(n)}(A)
\label{finoan}%
\end{equation}
holds true for any $A$ with $|A|\leq n$.  By Remark \ref{richiamare}and due to
\eqref{UltimeEpsNulle}, we again obtain that for any $A\subseteq[m]$ with
$|A|>n$
\begin{equation}
\alpha_{i}^{(n)}(A)=\frac{1}{|A|}, \label{alfauguali2}%
\end{equation}
regardless of the index $i\in A$.

We now want to show that also the claim $\mathcal{H}(n+1)$ holds. On this
purpose, recalling the condition $\varepsilon^{(n+1)}(n+2)=\cdots
=\varepsilon^{(n+1)}(m)=0$, we set
\[
\varepsilon^{(n+1)}(\ell)=\varepsilon^{(n)}(\ell), \text{ for }
\ell=2,\ldots,n.
\]
It remains to conveniently choose the positive value $\varepsilon
^{(n+1)}(n+1)$ in order to get
\begin{equation}
\sigma(A,i)<\sigma(A,j)\Rightarrow\alpha_{i}^{(n+1)}(A)>\alpha_{j}^{(n+1)}(A),
\label{finoan+1}%
\end{equation}
for any $A$ with $|A|\leq n+1$.

The existence of such a value $\varepsilon^{(n+1)}(n+1)>0$ will follow from
the inductive hypothesis $\mathcal{H}(n)$ and from the fact the probabilities
$(\alpha_{i}^{(n+1)}(A))$ are continuous with respect to the collection of the
intensities $(\mu_{i}^{(n+1)}(\cdot))$ and therefore they are also continuous
with respect to the collection of the coefficients $\varepsilon^{(n+1)}(\ell)$
($\ell=2,...,m$). Thus, by continuity, the validity of \eqref{finoan+1} is
maintained for any $A$ with $|A|\leq n$, provided that $\varepsilon
^{(n+1)}(n+1)>0$ is sufficiently small.

Notice that the assumption that all the inequalities in \eqref{finoan} are
strict is unavoidable here.

Let us now consider the sets $A$ with cardinality $|A|=n+1$. Having selected
such a value $\varepsilon^{(n+1)}(n+1)>0$, by \eqref{formulata}, we obtain
\[
\alpha_{i}^{(n+1)}(A)=\frac{1}{m}+\sum_{k=1}^{m-n-2}\frac{m-n-1}{m}%
\frac{m-n-2}{m-1}\ldots\frac{m-n-k}{m-k+1}\frac{1}{m-k}+
\]%
\begin{equation}
+\frac{m-n-1}{m}\frac{m-n-2}{m-1}\ldots\frac{1}{n+2}\frac{1-(\sigma
(A,i)-1)\varepsilon^{(n+1)} (n+1)}{\left[  \ell-n(n+1)\varepsilon^{(n+1)}
(n+1)\right]  }. \label{ASidibNUOVA}%
\end{equation}
This formula shows that if $\sigma(A,i) > \sigma(A,j) $ then $\alpha_{i} (A )
< \alpha_{j} (A)$. Thus \eqref{finoan+1} is satisfied also for $A$ with
$|A|=n+1$ and the thesis is proven by induction.
\end{proof}

\begin{rem}\label{razionali} The quantities $\varepsilon (2), \ldots , \varepsilon (m)$, appearing within the above proof, are only required to be sufficiently small. More precisely, $\varepsilon (n)$ is only required to belong to a suitable right neighbourhood of zero. Such a neighbourhood  will possibly  depend on  $\varepsilon (2), \ldots , \varepsilon (n-1)$. Therefore one can select
$\varepsilon (2), \ldots , \varepsilon (m)\in \mathbb{Q}_+$. With such a choice also the quantities
$\mu_i([m] \setminus A)$, defined in \eqref{scelta1Nv}, turn out to be rational, for any $A \subseteq [m]$ and $i \in A$. Finally, by \eqref{dispoNv}, one obtains that all the elements of $\mathbf{P_J}$ are  rational
as well.
\end{rem}

\end{document}